\documentclass[12pt]{article}

\usepackage{bbm}
\usepackage[dvips]{graphics} 
\usepackage[usenames]{color}

\usepackage{theorem}
\usepackage{stmaryrd}
\usepackage{times}

\usepackage{graphicx,amsmath,amssymb,color,subfigure,mathrsfs}
\usepackage{latexsym}
\usepackage{pgfplots}
\usepackage{hyperref,pdfsync}

\newcommand{\lnT}{\ell_{\eps}}

\newtheorem{theorem}{Theorem}[section]
\newtheorem{lemma}[theorem]{Lemma}
\newtheorem{proposition}[theorem]{Proposition}

\def\blacksquare{
\thinspace\nobreak \vrule width 5pt height 5pt depth 0pt}
\newtheorem{remark}[theorem]{Remark}

\newenvironment{proof}{\begin{trivlist}
                      \item[]\hspace{0cm}{\bf Proof: }
                       \hspace{0cm} }{\hfill $\blacksquare$
                     \end{trivlist}}

\makeatletter

\@addtoreset{equation}{section}
\makeatother
\def\d{\; {\rm d}}

\def\O{{\cal O}}
\def\B{{\cal B}}

\def\cprime{$'$}

\def\R{\mathbb R}

\def\N{\mathbb N}
\def\Z{\mathbb Z}
\def\C{\mathbb C}
\def\sL{{\rm L}}
\def\sH{{\rm H}}

\def\Fc{{\cal F}}
\def\Gc{{\cal G}}

\def\Mc{{\cal M}}
\def\f{\mathsf f}

\def\eps{\varepsilon}

\def\dom2inf#1{\R^2\backslash#1}
\def\domzero{\Omega_0}
\def\domeps{\Omega_\eps}

\def\uzero{u_0}
\def\ueps{u_\eps}
\def\vzero{v_0}
\def\veps{v_\eps}
\def\rest{r_\eps}
\def\reste#1{r_\eps^#1}

\def\heps{h_{\eps}}

\def\W{{\mathfrak w}}

\def\domieps#1{\omega^#1_{\eps}}
\def\xieps#1{x_\eps^#1}
\def\domi#1{\omega^#1}

\def\Tc{{\cal T}}

\title{Interactions between moderately close inclusions\\ for the 2D Dirichlet-Laplacian}
\author{V. Bonnaillie-No\"el, M. Dambrine and C. Lacave}
\def\adrese{
\begin{description}
\item[V. Bonnaillie-No\"el:] D\'epartement de Math\'ematiques et Applications (DMA UMR 8553), Paris Sciences et Lettres, CNRS, ENS Paris, 45 rue d'Ulm, F-75230 PARIS cedex 05, France.\\
Email: \texttt{bonnaillie@math.cnrs.fr}\\
Web page: \url{http://www.math.ens.fr/~bonnaillie/}
\item[M. Dambrine:] 
Laboratoire de Math\'ematiques et de leurs Applications, UMR 5142, CNRS, Universit\'e de Pau et des Pays de l'Adour, av. de l'Universit\'e, BP 1155, F-64013 Pau Cedex, France.\\
Email: \texttt{marc.dambrine@univ-pau.fr}\\
Web page: \url{http://web.univ-pau.fr/~mdambrin/}
\item[C. Lacave:] Univ Paris Diderot, Sorbonne Paris Cit\'e, Institut de Math\'ematiques de Jussieu-Paris Rive Gauche, 
UMR 7586, CNRS, Sorbonne Universit\'es, UPMC Univ Paris 06, F-75013, Paris, France.\\
Email: \texttt{christophe.lacave@imj-prg.fr}\\
Web page: \url{https://www.imj-prg.fr/~christophe.lacave/}
\end{description}
}

\begin{document}

\maketitle

\abstract{
This paper concerns the asymptotic expansion of the solution of the Dirichlet-Laplace problem in a domain with small inclusions. This problem is well understood for the Neumann condition in dimension greater or equal than two or Dirichlet condition in dimension greater than two. The case of two circular inclusions in a bidimensional domain was considered in \cite{BD13}. In this paper, we generalize the previous result to any shape and relax the assumptions of regularity and support of the data. Our approach uses conformal mapping and suitable lifting of Dirichlet conditions. We also analyze configurations with several scales for the distance between the inclusions (when the number is larger than $2$). \\
} 


\section{Introduction}

In many application fields ranging from electrical engineering to flow around obstacles, one has to consider Poisson equation in a domain presenting holes.
The presence of small inclusions or of a surface defect modifies the solution of the Laplace equation posed in a reference domain $\domzero$. If the characteristic size of the perturbation is small, then one can expect that the solution of the problem posed on the perturbed geometry is close to the solution of the reference shape. Asymptotic expansion with respect to that small parameter --the characteristic size of the perturbation-- can then be performed. 

The case of a single inclusion $\omega$, centered at the origin $0$ being either in $\domzero$ or in $\partial\domzero$, has been deeply studied, see~\cite{MazyaNazarovPlamenevskijT1,MazyaNazarovPlamenevskijT2,Ilin,LewinskiSokolowski,NazarovSokolowski,DambrineVial,DambrineVial2}. 
The techniques rely on the notion of {\em profile}, a normalized solution of the Laplace equation in the exterior domain obtained by {\em blow-up} of the perturbation. It is used in a fast variable to describe the local behavior of the solution in the perturbed domain. For Neumann boundary conditions in dimension greater or equal than two and Dirichlet boundary condition in dimension greater than two, convergence of the asymptotic expansion is obtained thanks to the decay of the {\em profile} at infinity.

The case of some inclusions was considered for example in the series of papers of A. Movchan and V. Maz'ya \cite{MazyaMovchan08,MazyaMovchan10} where an asymptotic approximation of Green's function is built and justified in a domain with many inclusions. The points where the inclusions are shunk are fixed in those works. In \cite{BDTV09}, the Neumann case where the distance between the holes tends to zero but remains large with respect to their characteristic size was investigated for two perfectly insulated inclusions: a complete multiscale asymptotic expansion of the solution to the Laplace equation is obtained in a three scales case. 

Recently, V. Bonnaillie-No\"el and M. Dambrine have considered in \cite{BD13} the case of two circular defects with homogeneous Dirichlet boundary conditions in a bidimensional domain. They distinguish the cases where the distance between the object is of order 1 and the case where it is larger than the characteristic size of the defects but small with respect to the size of the domain. They have derived the complete expansion and built a numerical method to solve the problem. Our aim is to extend their result to any geometry for the inclusion as well as to richer geometric configurations. Note that L. Chesnel and X. Clayes have proposed an alternative numerical scheme in \cite{CC14}. 

Let us make precise the problem under consideration. Let $\domzero$ be a bounded simply connected $C^2$ domain in $\mathbb{R}^2$. We consider $N$ inclusions of size $\eps$: for any $i=1,\dots,N$ let $\domieps{i}:= \xieps{i}+ \eps \domi{i}$ where $\domi{i}$ is a bounded simply connected open $C^2$ set containing $0$.
The centers of the inclusions $\xieps{i}$ are distinct, belong to $\domzero$ and tend to $x^{i}_{0} \in \domzero$ as $\eps \to 0$. We allow to have the same limit position $x^{i}_{0}= x^{j}_{0}$ for some $i\neq j$ but we assume that the distance for any $\eps>0$ is much larger than the size of the inclusion (see \eqref{hypo:eta} in Section~\ref{sec:4} for precise definitions). 
 Therefore, for $\eps$ small enough, we have $\domieps{i} \subset \domzero$ and $\overline{\domieps{i}}\cap \overline{\domieps{j}}=\emptyset$ for any $i\neq j$.
The domain depending on $\eps$ is then : 
\[
\domeps := \domzero \setminus \Big( \bigcup_{i=1}^N \overline{\domieps{i}}\Big).
\]
We search an asymptotic expansion for the solution $u_{\eps}$ of the Dirichlet-Laplace problem in $\domeps$:
\begin{equation}\label{Laplace-eps}
\left\{\begin{array}{rclcl}
-\Delta \ueps&=& \f & \mbox{ in }& \domeps ,\\
\ueps &=& 0 & \mbox{ on }& \partial \domeps,
\end{array}\right.
\end{equation}
where $\f\in\sH^{-1+\mu}(\domzero)$ with $\mu>0$. Since the $\sH^1$-capacity of $\domieps{i}$ tends to $0$ as $\eps$ tends to $0$ (see e.g. \cite{HP05} for all details on the capacity), it is already clear that $\ueps$ converges strongly in $\sH^1_{0}(\domzero)$ to $\uzero$ which is the unique solution of the Dirichlet-Laplace problem in $\domzero$:
\begin{equation}\label{Laplace-0}
\left\{\begin{array}{rclcl}
-\Delta \uzero&=& \f & \mbox{ in }& \domzero ,\\
\uzero &=& 0 & \mbox{ on }& \partial \domzero.
\end{array}\right.
\end{equation}
In this convergence, we have extended $\ueps$ by zero inside the inclusions.
At order $0$, we thus have $\ueps = \uzero + \reste0$ where $\|r^0_{\eps}\|_{\sH^1(\domzero)} \to 0$ as $\eps\to 0$. Let us notice that the remainder $\reste{0}$ satisfies
\begin{equation}\label{eq.reste0}
\begin{cases}
-\Delta \reste{0} =0 & \mbox{ in }\domeps,\\
\reste0 = 0 & \mbox{ on }\partial\domzero,\\
\reste0 = -\uzero & \mbox{ on }\partial\domieps{i},\quad\forall 1\leq i\leq N.
\end{cases}
\end{equation}
The purpose of this paper is to get the lower order terms in the expansion of $\ueps$. We follow the leading ideas of \cite{BD13} and generalize their results to any geometry of the defects while simplifying the proofs. We also relax the assumption on the regularity of $\f$ and do not more assume that $\f$ is supported far away the inclusions. The drawback is that the correctors built in this work are less explicit but remain numerically computable. \\

The paper is organized as follows. In Section~\ref{sec:2}, we present the basic tools essentially of complex analysis required for the sequel. Section~\ref{sec.3.1} is devoted to the presentation of the strategy in the case of a single inclusion, while Section~\ref{sec:4} deals with a finite number of holes separated by various distance. 

\section{Basic tools}\label{sec:2}

In this section, we introduce material which will be used in the following sections. The explicit solution to the Laplace problem is well known in the full plane, and we can also find in literature the Green function in or outside the unit disk. Conformal mapping is a convenient change of variable for the Laplace problem in order to get a formula inside or outside any simply connected compact set. In the sequel, we identify $\R^2$ and $\C$ through $(x_{1},x_{2})=x_{1}+i x_{2}=z$.

\subsection{Conformal mapping and Green function}

By the Riemann mapping theorem, there exists a unique biholomorphism $\Tc^i$ from $\R^2\setminus\overline{\domi{i}}$ to $\overline{\B(0,1)}^c$ such that
\[
\Tc^i(\infty)=\infty \qquad \text{and}\qquad {\rm arg}(\Tc^i)'(\infty)=0,
\]
which reads in the Laurent series decomposition as
\begin{equation}\label{eq.TLaurent}
\Tc^i(z)=\beta^i z + \sum_{k\in \N} \frac{\beta^i_{k}}{z^k} ,\qquad \forall z\in \B(0,R)^c,
\end{equation}
with $\beta^i \in \R^+_{*}$ is called the transfinite diameter (or logarithmic capacity) of $\omega^{i}$. In the previous equality, the radius $R$ is chosen large enough such that $\domi{i}\subset \B(0,R)$. A consequence of such a decomposition is the existence of a constant $C>0$ such that:
\begin{equation}\label{T-bz}
\| \Tc^i(z)-\beta^i z \|_{\sL^\infty(\R^2 \setminus \overline{\domi{i}})} \leq C. 
\end{equation}

For the Dirichlet problem in an open set $\Omega$, the Green function is a function $G_{\Omega}$ from $\Omega\times \Omega$ to $\R$ such that
\[
G_{\Omega}(x,y)=G_{\Omega}(y,x),\qquad G_{\Omega}(x,y)=0 \text{ if }x\in \partial \Omega,\qquad \Delta_{x} G_{\Omega}(x,y) =\delta(x-y),
\]
where $\delta$ is the Dirac measure centered at the origin. In the full plane, we have $G_{\R^2}(x,y)=\frac1{2\pi} \ln |x-y|$ and outside the unit disk we have 
\[
G_{\overline{\B(0,1)}^c}(x,y)=\frac1{2\pi} \ln \frac{|x-y|}{|x-y^*| |y|},
\]
with the notation
\[
y^* = \frac{y}{|y|^2}.
\]
Thanks to the conformal mapping $\Tc^i$ we deduce the Green function in the exterior of $\domi{i}$:
$$
G_{\R^2\setminus\overline{\domi{i}}}(x,y)=\frac1{2\pi} \ln \frac{|\Tc^i(x)-\Tc^i(y)|}{|\Tc^i(x)-\Tc^i(y)^*| |\Tc^i(y)|}.
$$
Inside $\domzero$, let us introduce the function $\mathcal G$ defined by
\begin{equation}\label{eq.Gphi}
\mathcal G[\varphi](x):=\int_{\partial\domzero}\partial_{n}G_{\domzero}(x,y)\varphi(y) \d\sigma(y),\qquad \forall x\in\domzero,
\end{equation}
where $G_{\domzero}$ can be defined as in $G_{\R^2\setminus\overline{\domi{i}}}$ replacing $\Tc^i$ by $\Tc^0$ a biholomorphism from $\domzero$ to $\B(0,1)$.
This function satisfies
\begin{equation*}\begin{cases}
-\Delta \mathcal G[\varphi]=0 & \mbox{ in }\domzero,\\
\mathcal G[\varphi]=\varphi& \mbox{ on }\partial\domzero.
\end{cases}\end{equation*}

\begin{remark}
 For any shape $\domzero$ and $\domi{i}$ we can compute numerically $\Tc^0$, $\Tc^{i}$, only by solving once a Laplace problem. This is discussed in \cite[Chapter 16, section 5]{Henrici}. We can also make explicit the conformal mapping for some geometries. \\
When $\domi{i}$ is an ellipse $\{ (x,y)\in\R^2,\ a^2 x^2 + b^2 y^2 < c^2\}$ with $a,b,c>0$, then we find 
$$(\Tc^{i})^{-1}(z)= \frac{c}a\frac{z+1/z}2 + i \frac{c}b\frac{z-1/z}{2i},\qquad{\rm and }\qquad \Tc^{i}(z)=\frac{c^{-1} z \pm \sqrt{c^{-2}z^2+b^{-2}-a^{-2}}}{a^{-1}+b^{-1}}.$$
Even if it is not in the geometrical setting of this paper, let us notice that if $\domi{i}$ is a segment, then the Joukowski function $h(z)=(z+1/z)/2$ gives an explicit formula for $\Tc^i$. Indeed, $h$ maps the exterior of the unit disk to the exterior of the segment $[-1,1]\times\{0\}$. Then, up to a translation, dilation and rotation, we find that $\Tc^{i}(z)=z\pm\sqrt{z^2-1}$ (where $\pm$ is chosen in order that $|\Tc^{i}(z)|>1$, depending on the definition of $\sqrt{\ }$) sends the exterior of the segment to the exterior of the unit disk. \\ 
\end{remark}

\subsection{Dirichlet problem in an exterior domain}

Even if the main goal of the Green function is to produce an explicit solution of the Laplace problem, we only use these functions to prove the following Lemma.

\begin{lemma}\label{lemme:relevement:domaine:exterieur}
Let $\omega$ be a bounded simply connected open $C^2$ subset of $\R^2$. If $F\in \sH^{1/2}(\partial \omega)$, then the boundary value problem
\begin{equation*}
\begin{cases}
-\Delta \Psi =0 &\text{ in } \R^2\setminus \overline{\omega},\\ 
\Psi=F &\text{ on } \partial \omega,
\end{cases}
\end{equation*}
admits a unique weak solution $\Psi$ in the variational space
\[
\sH^1_{\rm log}=\Big\{ \Phi; \ \frac{\Phi}{(1+|X|)\ln (2+|X|)}\in \sL^2(\R^2\setminus\overline{\omega})\ \text{ and }\ \nabla \Phi \in \sL^2(\R^2\setminus\overline{\omega}) \Big\}.
\]
Furthermore, we have the following properties:
\begin{enumerate}
\item\label{Point1} For any $n\in \N$, the solution $\Psi$ can be decomposed as
\begin{equation}\label{VLaurent}
\Psi(X)= \sum_{k=0}^n \Psi_{k}(X) + R_{n+1}(X)
\end{equation}
where $\Psi_{0}$ is constant, $\Psi_{k}(r,\theta)=\cfrac{a_{k}\cos (k\theta)+b_{k}\sin(k\theta)}{r^k}$ and 
\begin{equation}\label{eq.Rn}
|X|^n |R_{n}(X)| \leq C_{n}\|F\|_{\sL^\infty(\partial\omega)}
\end{equation}
for some $C_{n}>0$ independent of $F$.
\item\label{Point2} Let $\Tc$ be the biholomorphism from $\R^2\setminus \overline{\omega}$ onto the exterior of the unit disk (such that $\Tc(\infty)=\infty$ and $\arg\Tc'(\infty)=0$), we have
\begin{equation}\label{eq.Psi0}
\Psi_{0} = \Psi_{0}[F] 
= \frac1{2\pi} \int_{0}^{2\pi} F \circ \Tc^{-1} \begin{pmatrix} \cos \theta \\ \sin \theta \end{pmatrix} \d\theta 
= \frac1{2\pi} \int_{\partial \omega} F(Y) \sqrt{\det D\Tc (Y)} \d \sigma(Y).
\end{equation}
\item\label{Point3} $\Psi_{0}=0$ if and only if there exist $R>0$ and $\widehat \Psi \in \mathcal{H}(\R^2\setminus \B(0,R))$ such that 
$$\Psi ={\rm Re\ }\widehat \Psi\quad \mbox{ and }\quad \int_{\partial \B(0,R)} \frac{\widehat \Psi}z \d z =0.$$
\end{enumerate}
\end{lemma}

\begin{proof}
The well-posedness in the variational space is a standard result coming from Lax-Milgram theorem (see e.g. \cite{GiroireThese}).

Next, we note that $\nabla \Psi =\begin{pmatrix} \partial_{1} \Psi \\ \partial_{2} \Psi\end{pmatrix}$ is divergence and curl free, so 
$$\widehat{\nabla \Psi} (z):= \partial_{1} \Psi(x,y)-i \partial_{2} \Psi(x,y)$$ 
is holomorphic because it verifies the Cauchy-Riemann equations. Hence, $\widehat{\nabla \Psi}$ admits a Laurent series decomposition on $\B(0,R)^c$, with $R$ such that $\omega\subset \B(0,R)$, and as $\nabla \Psi$ is square integrable, we deduce that
\[
\widehat{\nabla \Psi} (z) = \sum_{k=2}^{+\infty} \frac{c_{k}}{z^k},\quad \text{for all } z\in \B(0,R)^c.
\]
Obviously, the function
$$g(z):=\sum_{k=1}^{+\infty} \frac{-c_{k+1}}{kz^k}$$ 
is a holomorphic primitive of $\widehat{\nabla \Psi}$. Decomposing the function $g$ in real and imaginary part $$g(z)=g_{1}(x,y) + i g_{2}(x,y),$$ we verify that
\[
\frac{{\rm d} g}{{\rm d} z}(z) 
=\tfrac12\big( \partial_{1}g_{1} + \tfrac1{i}\partial_{2}g_{1} \big)+\tfrac{i}{2}\big( \partial_{1}g_{2} + \tfrac1{i}\partial_{2}g_{2} \big)
=\partial_{1}g_{1} -i \partial_{2}g_{1}, 
\]
where we have used the Cauchy-Riemann equations on $g$. Therefore, there exists $\Psi_{0}\in \R$ such that
\begin{equation}\label{Laurent-Psi}
 \Psi(x,y)=\Psi_{0}+ g_{1}(x,y)=\Psi_{0} + {\rm Re\ }g(z)= \Psi_{0}+ {\rm Re\ }\Big(\sum_{k=1}^{+\infty} \frac{-c_{k+1}}{kz^k}\Big) = \Psi_{0}+ \sum_{k=1}^{+\infty} \Psi_{k}(x),
\end{equation}
with 
$$a_{k} = - {\rm Re}\frac{c_{k+1}}k,\qquad \mbox{ and }\qquad b_{k} = - {\rm Im}\frac{ c_{k+1}}k.$$ 
This ends the proof of the decomposition of Point~\ref{Point1}. We establish \eqref{eq.Rn} at the end of this proof. 

We use \eqref{Laurent-Psi} to prove the third point. If $\Psi_{0}=0$, then $\Psi(x,y)={\rm Re\ } g(z)$ with $g\in \mathcal{H}(\C \setminus \B(0,R))$ and we compute by the Cauchy residue theorem that 
$$\int_{\partial \B(0,R)} \frac{g}z \d z =0.$$ 
At the opposite, let us assume that there exist $R$ and $\widehat \Psi \in \mathcal{H}(\C\setminus \B(0,R))$ such that 
$$\Psi ={\rm Re\ }\widehat \Psi\qquad \mbox{ and }\qquad \int_{\partial \B(0,R)} \frac{\widehat \Psi}z \d z =0.$$ 
We decompose $\widehat \Psi$ in Laurent series $\widehat \Psi = \sum_{k\in \Z} {d_k}{z^{-k}}$, and the condition 
$$\int_{\partial \B(0,R)} \frac{\widehat \Psi}z \d z =0$$ 
reads as $d_{0}=0$. As $\Psi ={\rm Re\ }\widehat \Psi$, we conclude that $\Psi_{0}={\rm Re\ }d_{0}=0$. This establishes Point~\ref{Point3}.

Concerning the second statement, we write the expression of the solution $\Psi$ in terms of the Green function:
\begin{equation}\label{eq.Green}
\Psi(X) = \int_{\partial \omega} \partial_{n} G(Y,X) F(Y) \d \sigma(Y),
\end{equation}
where $n$ is the outgoing normal vector of $\R^2\setminus \overline{\omega}$. Thanks to the explicit formula of $G$, we compute:
\[
\Psi(X) = \frac1{2\pi}\int_{\partial \omega} F(Y) n(Y) \cdot D\Tc^T(Y) \Big(\frac{\Tc(Y)-\Tc(X)}{|\Tc(Y)-\Tc(X)|^2} -\frac{\Tc(Y)-\Tc(X)^*}{|\Tc(Y)-\Tc(X)^*|^2} \Big) \d \sigma(Y).
\]
Indeed, the assumption $\partial \omega\in C^2$ allows us to apply the Kellogg-Warschawski theorem (see \cite[Theorem 3.6]{Pomm}) to state that $D\Tc^{-1}$ is continuous up to the boundary.

When $X\to \infty$, we note that $\Tc(X)\to \infty$ and $\Tc(X)^*\to 0$, hence it is clear that
\[
\Psi_{0} = - \frac1{2\pi}\int_{\partial \omega} F(Y) n(Y) \cdot D\Tc^T(Y)\frac{\Tc(Y)}{|\Tc(Y)|^2} \d \sigma(Y).
\]
A parametrization of $\partial \omega$ can be 
$$Y= \gamma (\theta)=\Tc^{-1}\begin{pmatrix} \cos \theta \\ \sin \theta \end{pmatrix}\quad\mbox{ for }\quad\theta \in [0,2\pi].$$ 
Hence, $ (\gamma' (\theta))^{\perp} = -D\Tc^{-1}(\Tc(Y)) \Tc(Y)$ where we have used the Cauchy-Riemann equations. We compute again by the Cauchy-Riemann equations that
\[
| \gamma' (\theta)|^2 = \Tc(Y) \cdot (D\Tc^{-1}(\Tc(Y)))^T D\Tc^{-1}(\Tc(Y)) \Tc(Y)= \det D\Tc^{-1}(\Tc(Y)) | \Tc(Y) |^2 = \frac1{ \det D\Tc(Y)}.
\]
We deduce that 
\begin{equation*}\begin{split}
n(Y)= &\frac{(\gamma' (\theta))^{\perp}}{| \gamma' (\theta)|}= - \sqrt{\det D\Tc(Y)} D\Tc^{-1}(\Tc(Y)) \Tc(Y)= - \sqrt{\det D\Tc(Y)} (D\Tc(Y))^{-1} \Tc(Y)\\
=&- \sqrt{\det D\Tc(Y)} \frac{(D\Tc(Y))^T}{\det D\Tc(Y)} \Tc(Y)=- \frac{(D\Tc(Y))^T}{\sqrt{\det D\Tc(Y)} }\Tc(Y) .
\end{split}\end{equation*}
Hence, we get
\[
n(Y) \cdot D\Tc^T(Y)\frac{\Tc(Y)}{|\Tc(Y)|^2} =
- \frac{ \Tc(Y)}{\sqrt{\det D\Tc(Y)}} \cdot D\Tc(Y) D\Tc^T(Y)\frac{\Tc(Y)}{|\Tc(Y)|^2} = -\sqrt{\det D\Tc(Y)}.
\]
This allows to conclude:
\begin{align*}
\Psi_{0} &= \frac1{2\pi}\int_{\partial \omega} F(Y) \sqrt{\det D\Tc(Y)} \d \sigma(Y)\\
&= \frac1{2\pi}\int_0^{2\pi} F(\gamma(\theta)) \sqrt{\det D\Tc(\gamma(\theta))} | \gamma' (\theta)| \d \theta 
= \frac1{2\pi} \int_{0}^{2\pi} F \circ \Tc^{-1} \begin{pmatrix} \cos \theta \\ \sin \theta \end{pmatrix} \d \theta .
\end{align*}
An important consequence of this equality is the following estimate:
\begin{equation}\label{rem.2.2}
 |\Psi_0[F] | \leq \|F\|_{\sL^\infty(\partial \omega)}.
\end{equation}

Now, we note by the Laurent series decomposition \eqref{Laurent-Psi} that there exists $R$ large enough such that $\| \Psi(X) \| \leq |\Psi_0[F] | + \|F\|_{\sL^\infty(\partial \omega)}$ for all $X\in\B(0,R)^c$. Therefore, $ \Psi$ is bounded by $2\|F\|_{\sL^\infty(\partial \omega)}$ outside this ball, and in $\B(0,R)\setminus \overline{\omega}$ we use the maximum principle to state that
\[
\| \Psi(X) \| \leq \| \Psi \|_{\sL^\infty(\partial\omega\cup \partial \B(0,R))} \leq 2 \|F\|_{\sL^\infty(\partial \omega)} \qquad \forall X \in \B(0,R)\setminus \overline{\omega},
\]
hence
\[
\| \Psi \|_{\sL^\infty(\R^2\setminus \overline{\omega})}\leq 2\|F\|_{\sL^\infty(\partial \omega)}.
\]
Next, we consider $R_{0}>1$ such that $\omega\subset \B(0,R_{0}-1)$. Since $\Psi$ is harmonic, the mean value formula implies
\[\| \nabla \Psi \|_{\sL^\infty(\partial \B(0,R_{0}))} \leq 2 \| \Psi \|_{\sL^\infty( \B(0,R_{0}-1)^c)}\leq 4 \|F\|_{\sL^\infty(\partial \omega)}.\]
Combine this estimate with the Cauchy formulas gives
\[
|c_{k}|=\Big| \frac1{2i\pi} \int_{\partial \B(0,R_{0})} \widehat{\nabla \psi}(z) z^{k-1} \d z\Big| \leq 4 R_{0}^k \|F\|_{\sL^\infty(\partial \omega)}.
\]
Therefore, for any $|X|\geq 2 R_{0}$ and any $n\geq 1$, we have
\[
 |X|^n |R_{n}(X)| \leq \sum_{k=n}^\infty \frac{|c_{k+1}|}{k |X|^{k-n}} \leq 4 \| F \|_{\sL^\infty(\partial \omega)} \sum_{k=n}^\infty \frac{R_{0}^{k+1}}{ (2R_{0})^{k-n}} \leq 8 R_{0}^{n+1} \| F \|_{\sL^\infty(\partial \omega)} .
 \]
In $\B(0,2R_{0})\setminus \overline{\omega}$, it is clear that
\[
 |X|^n |R_{n}(X)| \leq |X|^n |\Psi (X) |+ |X|^n |\Psi_{0} | + \sum_{k=1}^{n-1} \frac{|c_{k+1}|}{k} |X|^{n-k}\leq C(n,R_{0}) \| F \|_{\sL^\infty(\partial \omega)}.
\]
This ends the proof of \eqref{eq.Rn}.
\end{proof}

\begin{remark}
The previous lemma is still available assuming less regularity on $\omega$ than $C^2$. 
To write a representation formula, it is enough to assume that $\omega$ is $C^{1,\alpha}$ with a finite number of corners with openings in $(0,2\pi)$ \cite[p. 20]{HW08}. 
We also used the maximum principe in $(\B\setminus\overline{\omega})$, this requires that $(\B\setminus\overline{\omega})$ is the interior ball property that is $\overline{\omega}$ has the exterior ball property \cite[Chapter 3]{GT}. This requires that only corners with opening in $(0,\pi)$ could be considered.\\
If $\omega$ is less regular (with crack, for example), the main difficulty is to establish \eqref{eq.Psi0} or \eqref{eq.Green}, see \cite{CoDaDu03}. 
Nevertheless, to justify our construction for any inclusion, we only need \eqref{VLaurent} and \eqref{eq.Rn}.
In the case of domain with cracks, decomposition \eqref{VLaurent} is still valid but it requires more sophistical tools to prove estimates like \eqref{eq.Rn} (see \cite{CoDaDu03}) and we do not want to enter in this special feature. 
\end{remark}

\begin{remark}
When $\omega$ is the unit ball $\B$, then we recover the statement of \cite[Lemma 2.1 (3)]{BD13} and 
$$\int_{\partial\omega} F=0\quad \Longrightarrow\quad \Psi_{0}=0.$$
The property of the zero mean value is crucial in \cite{BD13} in the case of circular inclusions : it implies the profile decay at infinity and allows to construct the terms of the expansion. 
In the present paper, even if the boundary condition has the zero mean value, the associated profile does not decay at infinity and we have to lift $\Psi_{0}$ suitably (see Section~\ref{sec.3.1}). 
 \end{remark}
We end this section by recalling the following classical elliptic estimate (see e.g. \cite[Theorem 2.10]{GT}).

\begin{lemma}\label{lem.GT}
Let $x_0\in \domzero$ and $\delta >0$ such that $\B(x_0,\delta)\subset \domzero$. Then, for any $n\in \N$, there exists $C_n(\delta)>0$ such that
\[
\| D^n u \|_{\sL^\infty(\B(x_0,\delta/2))} \leq C_n(\delta) \| u \|_{\sL^\infty(\B(x_0,\delta))}
\]
for any harmonic function $u$ (i.e. such that $\Delta u =0$ on $\B(x_0,\delta)$).
\end{lemma}

\section{One inclusion}\label{sec.3.1}

In the case of one inclusion, we omit the index $1$ and denote $\omega$, $x_{\eps}$, $\Tc_{\eps}$, $\ldots$
\subsection{One iteration}
To deal with Equation \eqref{eq.reste0} satisfied by $\reste0$, we consider the more general boundary values problem:
\begin{equation}\label{eq.veps}
\begin{cases}
-\Delta \veps = 0 &\text{ in } \domeps,\\
\veps = \varphi &\text{ on }\partial\domzero,\\
\veps = f &\text{ on }\partial\omega_{\eps}.
\end{cases}
\end{equation}
At the first order, $\veps$ is approximated by $\vzero={\mathcal G}[\varphi]$, see \eqref{eq.Gphi}, which is the solution of
\begin{equation*}
\begin{cases}
-\Delta \vzero = 0&\text{ in } \domzero,\\
\vzero = \varphi &\text{ on }\partial\domzero.
\end{cases}
\end{equation*}
In order to gain one order in the remainder of the asymptotic expansion of $\ueps$, we need to introduce some notations. 
For any function $f\in\sH^{1/2}(\partial\omega_{\varepsilon})$, we define $F\in \sH^{1/2}(\partial\omega)$ by
$$F(X) = f(x),\qquad\forall X=\frac{x-x_{\eps}}{\eps}\in \partial\omega.$$
By Lemma~\ref{lemme:relevement:domaine:exterieur}, there exists a unique function $\Psi\in \sH^1_{\rm log}$ solution of
\begin{equation}\label{eq.F}
\begin{cases}
-\Delta \Psi = 0& \text{ in } \R^2\setminus\overline{\omega},\\
\Psi=F & \text{ on }\partial\omega.
\end{cases}\end{equation}
This solution will be denoted by 
\begin{equation}\label{eq.Fcomega}
\Fc_{\omega}[F] := \Psi.
\end{equation}
Outside $\omega_{\eps}$, we can consider the associated exterior problem
\begin{equation*}
\begin{cases}
-\Delta \psi = 0& \text{ in } \R^2\setminus\overline{\omega_{\eps}},\\
\psi=f & \text{ on }\partial\omega_{\eps}.
\end{cases}\end{equation*}
Then we have for any $x\in\R^2\setminus\overline{\omega_{\eps}}$ 
\begin{equation*}
\Fc_{\omega_{\eps}}[f](x) := \psi(x)=\Psi(X)=\Fc_{\omega}[F](X),\qquad \mbox{ with }X=\frac{x-x_{\eps}}{\eps}\in\R^2\setminus \overline{\omega}.
\end{equation*}
The function $\Psi$ can be decomposed in polar coordinates (see \eqref{VLaurent}) as 
\begin{equation}\label{eq.PsiPsi}
\Psi(r,\theta)=\Psi_{0}+\widetilde{\Psi}(r,\theta)\quad \mbox{ with }\quad \widetilde{\Psi}(r,\theta)=\sum_{k\geq1} d_{k}(\theta) r^{-k}
\ \mbox{ and }\ d_{k}\in \text{Span}(\sin(k\cdot),\cos(k\cdot)).
\end{equation}
Note that $\widetilde \Psi=R_{1}$ and then satisfies estimate~\eqref{eq.Rn}. 
Let us define
\begin{equation}\label{Teps}
\Tc_{\eps}(x) := \eps \Tc\Big( \frac{x-x_{\eps}}{\eps}\Big)=\eps \Tc\Big(X\Big),
\end{equation}
which maps the exterior of $\overline{\omega_{\eps}}$ to the exterior of the disk $\overline{\B(0,\eps)}$. \\

Inspired by the case of the ball \cite{BD13}, we define the two main ingredients of our construction. 
The first one is adapted to lift a constant function on $\partial\omega_{\eps}$. In $\sH^1_{\rm log}$, the unique solution of the problem
$$
\begin{cases}
-\Delta v=0 &\mbox{ in }\R^2\setminus\overline{\omega_{\varepsilon}},\\ 
v=c&\mbox{ on }\partial\omega_{\varepsilon}
\end{cases}$$
is the constant function $v=c$. But this generates a constant term on $\partial\domzero$ and the unique solution of 
$$
\begin{cases}
-\Delta V=0 &\mbox{ in }\domzero,\\ 
V = c&\mbox{ on }\partial\domzero
\end{cases}$$
is the constant $c$ itself. Hence, we can't reduce the remainder with this procedure. \\
When $\omega$ is the unit ball, another way to lift the constant $c$ in $\R^2\setminus\overline{\omega_{\eps}}$ is to consider the function $x\mapsto c\frac{\ln|x-x_{\eps}|}{\ln\eps}$. 
For general inclusion $\omega$, we introduce the function
\begin{equation*}
\ell_{\eps}(x):=\ln{|\Tc_{\eps}(x)|},
\end{equation*}
which is a solution of
$$
\begin{cases}
-\Delta \ell_{\eps}=0 &\mbox{ in }\R^2\setminus\overline{\omega_{\varepsilon}},\\ 
\ell_{\eps}=\ln{\varepsilon}&\mbox{ on }\partial\omega_{\varepsilon}.
\end{cases}$$
As $\ell_{\eps}$ behaves at infinity like $\ln |x|$, this function does not belong to $\sH^1_{\rm log}$, but this allows us to find a non trivial harmonic extension of constant into $\overline{\omega_{\varepsilon}}^c$. 
The trace of the function $\lnT$ is non zero trace on the outer boundary $\partial\domzero$ and it is described by the behaviour at infinity of the Riemann mapping \eqref{eq.TLaurent}, and one gets thanks to \eqref{T-bz}
\begin{equation}\label{eq.lnT-lnbeta}
\| \lnT-\ln{\beta |\cdot-x_{\eps}|}\|_{\sL^\infty(\partial\domzero)}\leq M\varepsilon. 
\end{equation}
The second ingredient is to correct {\it the main order} of the traces on the outer boundary $\partial\domzero$. For this, we consider the function $w_{\omega}:=\Gc[\ln\beta|\cdot-x_{\eps}|]$ (see \eqref{eq.Gphi}) verifying
$$\begin{cases}
-\Delta w_{\omega}=0 &\text{ in }\domzero,\\ 
w_{\omega}=\ln{\beta |\cdot-x_{\eps}|}&\text{ on }\partial\domzero.
\end{cases}$$
Notice that the boundary condition $\ln{\beta |\cdot-x_{\eps}|}$ is $C^1$ on $\partial\domzero$, thus $w_{\omega}\in C^1(\domzero)$ (independently of $\omega$). 
Since $\domzero$ is a bounded domain in $\R^2$, we have by the maximum principle
\begin{equation}\label{eq.heps0}
\| w_{\omega}\|_{\sL^\infty(\domzero)}
\leq \| \ln{\beta |\cdot-x_{\eps}|}\|_{\sL^\infty(\domzero)} 
\leq |\ln (\beta{\rm diam}(\domzero))|,
\end{equation}
where ${\rm diam}(\partial\domzero)$ denotes the diameter of $\domzero$.

\begin{remark}
The function $w_{\omega}$ encodes two informations: the shape of $\omega$ via the transfinite diameter $\beta$ and the location of the inclusion {\it via} $x_{\eps}$.
\end{remark}

Unfortunately this function $w_{\omega}$ produces a trace on the small inclusion $\partial\omega_{\eps}$. 
The main idea of the construction is that an appropriate linear combination of the profiles $\lnT$ and $w_{\omega}$ allows to reduce the error on both $\partial\omega_{\varepsilon}$ and $\partial\domzero$.

\begin{proposition}
\label{prop:une:iteration}
Let $\veps$ be the solution of \eqref{eq.veps}. There is an harmonic function $\rest$ defined on $\Omega_{\varepsilon}$ such that
\begin{equation}
\label{eq.DLueps}
\veps(x)=\vzero(x)+\Fc_{\omega}[F]\left(\frac{x-x_{\eps}}{\varepsilon}\right)-\Psi_{0}[F]+\cfrac{\vzero(x_{\eps})-\Psi_{0}[F]}{w_{\omega}(x_{\eps})-\ln{\varepsilon}} \W_{\eps,\omega}(x)+\varepsilon\rest(x),
\end{equation}
with $\vzero={\mathcal G}[\varphi]$ defined in \eqref{eq.Gphi}, $\Psi_{0}[F]$ in \eqref{eq.Psi0}, $\Fc_{\omega}[F]$ in \eqref{eq.F}--\eqref{eq.Fcomega} and, with $\Tc_{\eps}$ given in \eqref{Teps}, 
$$\W_{\eps,\omega}:= \lnT-w_{\omega}\qquad\mbox{ with }\qquad
\lnT=\ln|\Tc_{\eps}|\quad\mbox{ and }\quad w_{\omega}=\Gc[\ln\beta|\cdot-x_{\eps}|].$$
Thus we have
$$\|\rest\|_{\sL^{\infty}(\partial\domzero \cap \partial\omega_{\varepsilon})} \leq 
C \left[ \| F\|_{\sL^{\infty}(\partial\omega)} + \| \varphi\|_{\sL^{\infty}(\partial\domzero)} \right].
$$
\end{proposition}

\begin{remark}
Thanks to \eqref{eq.heps0}, $w_{\omega}-\ln\eps$ does not vanish for $\eps$ small enough. So that \eqref{eq.DLueps} makes sense.
\end{remark}

\begin{proof}
The harmonicity of the remainder $\rest$ follows from the harmonicity of $\vzero$, $\Fc_{\omega}[F]$, $w_{\omega}$ and $\lnT$. In order to establish the $\sL^{\infty}$ estimate, we compute the traces of the function $\rest$ on each componant $\partial\omega_{\eps}$ and $\partial\Omega$ of the boundary. 

\noindent $\bullet$ On the outer boundary $\partial\domzero$, one has for any $x\in \partial\domzero$
\begin{align*}
\varepsilon\rest(x)
&=\varphi(x)-\left[\varphi(x)+\Fc_{\omega}[F]\left(\frac{x-x_{\eps}}{\varepsilon}\right) -\Psi_{0}+\cfrac{\vzero(x_{\eps})-\Psi_{0}}{w_{\omega}(x_{\eps})-\ln\varepsilon}\big( \lnT(x)-\ln{\beta|x-x_{\eps}|}\big) \right]\\
&=-\left[\widetilde \Psi\left(\frac{x-x_{\eps}}{\varepsilon}\right)+\cfrac{\vzero(x_{\eps})-\Psi_{0}}{w_{\omega}(x_{\eps})-\ln\varepsilon}\big( \lnT(x)-\ln{\beta|x-x_{\eps}|}\big) \right],
\end{align*}
with $\widetilde \Psi (=R_{1})$ defined in \eqref{eq.PsiPsi} and $\Psi_{0}=\Psi_{0}[F]$. 
Therefore, one has according to \eqref{eq.lnT-lnbeta} and \eqref{eq.Rn},
\[
\varepsilon\|\rest\|_{\sL^\infty(\partial\domzero)}
\leq C_{1}\varepsilon \|F\|_{\sL^{\infty}(\partial\omega)}+ \varepsilon h_{\varepsilon} M\left(|\vzero(x_{\eps})|+|\Psi_{0}|\right),\qquad
\mbox{ with }h_{\varepsilon}=\frac{1}{w_{\omega}(x_{\eps})-\ln\varepsilon}. 
\]
Since $|\Psi_{0}|\leq \|F\|_{\sL^{\infty}(\partial\omega)}$ (cf. \eqref{rem.2.2}) and $|\vzero(x_{\eps})|\leq \|\varphi\|_{\sL^{\infty}(\partial\domzero)}$ (by the maximum principle), we get
\[
\|\rest\|_{\sL^{\infty}(\partial\domzero)} \leq C\left[ \| F\|_{\sL^{\infty}(\partial\omega)} + h_{\varepsilon} \| \varphi\|_{\sL^{\infty}(\partial\domzero)} \right].
\]

\noindent $\bullet$ On the boundary of the inclusion $\partial\omega_{\eps}$, one has for $x\in\partial\omega_{\eps}$
\begin{align*}
\varepsilon\rest(x)
&= f(x)-\left[\vzero(x)+F\left( \frac{x-x_{\eps}}{\varepsilon}\right)-\Psi_{0}+\cfrac{\vzero(x_{\eps})-\Psi_{0}}{w_{\omega}(x_{\eps})-\ln\varepsilon}\left( \ln\varepsilon-w_{\omega}(x)\right) \right]\\
&= -\left[\vzero(x)-\Psi_{0}+\cfrac{\vzero(x_{\eps})-\Psi_{0}}{w_{\omega}(x_{\eps})-\ln\varepsilon}\left( \ln\varepsilon-w_{\omega}(x)\right) \right]\\
&= -(\vzero(x)-\vzero(x_{\eps}))
-\cfrac{\vzero(x_{\eps})-\Psi_{0}}{w_{\omega}(x_{\eps})-\ln\varepsilon}\left( w_{\omega}(x_{\eps})-w_{\omega}(x)\right).
\end{align*}
As soon as $\eps$ is small enough, there exists $\delta>0$ such that $\omega_{\eps}\subset\B(x_{\eps},\tfrac\delta2)\subset\B(x_{\eps},\delta)\subset \domzero$. 
Using the mean value formula and Lemma~\ref{lem.GT}, we get
\begin{align*}
\varepsilon\|\rest\|_{\sL^\infty(\partial\omega_{\eps})}
&\leq |x-x_{\eps}|\ \|\nabla \vzero\|_{\sL^\infty(\B(x_{\eps},\frac\delta2))} 
+\heps \left(|\vzero(x_{\eps})|+|\Psi_{0}|\right) |x-x_{\eps}|\ \|\nabla w_{\omega}\|_{\sL^\infty(\B(x_{\eps},\frac\delta2))} \\
&\leq C \left[\varepsilon \| \varphi\|_{\sL^{\infty}(\partial\domzero)}+ \varepsilon h_{\varepsilon} \| F\|_{\sL^{\infty}(\partial\omega)} \right].
\end{align*}
\end{proof}

\begin{remark}
As in \cite{BD13}, the coefficients $(\alpha,\beta)$ in front of $\lnT$ and $w_{\omega}$ in \eqref{eq.DLueps} are uniquely determined 
when we try to reduce the trace boundaries on $\partial\domzero$ and $\partial \omega_{\eps}$. This gives respectively the two equations of the following system
$$
\begin{cases}
\alpha+\beta =0,\\
\alpha\ln\varepsilon+\beta w_{\omega}(x_{\eps})=-\vzero(x_{\eps})+\Psi_{0},
\end{cases}
$$
so that $\alpha=-\beta=\cfrac{\vzero(x_{\eps})-\Psi_{0}}{w_{\omega}(x_{\eps})-\ln\varepsilon}$.
The scale $h_{\eps}=\frac{1}{w_{\omega}(x_{\eps})-\ln\varepsilon}$ is analogous to those appearing in the case of circular inclusion in \cite[Relation (2.6)]{BD13}.
\end{remark}

\subsection{Recursive construction}

Let us come back to the initial problems~\eqref{Laplace-eps}--\eqref{Laplace-0}. As previously, we consider $\reste0=\ueps-\uzero$ which satisfies \eqref{eq.reste0}. This problem has the form~\eqref{eq.veps} with 
$$\varphi=0,\qquad f=-\uzero.$$
Thus the function $\vzero={\mathcal G}[\varphi]=0$. 
Recall that we denote $F(X)=f(x)$ for any $x=x_{\eps}+\eps X\in\partial\omega_{\eps}$. 
Let us consider $\Psi_{0}=\Psi_{0}[F]$ determined in Lemma~\ref{lemme:relevement:domaine:exterieur}.
Applying Proposition~\ref{prop:une:iteration}, we have 
\begin{equation}\label{eq.restetilde1}
\ueps(x)-\uzero(x)=\reste0(x)
= \Fc_{\omega}[F]\left(\tfrac{x-x_{\eps}}{\varepsilon}\right)-\Psi_{0}-\frac{\Psi_{0}}{w_{\omega}(x_{\eps})-\ln{\varepsilon}} \ \W_{\eps,\omega}(x)+\varepsilon r_{\eps}^1(x),
\end{equation}
and
$$\| r_{\eps}^1\|_{\sL^{\infty}(\partial\domzero \cap \partial\omega_{\varepsilon})} \leq 
C \| \uzero\|_{\sL^{\infty}(\partial\omega_{\varepsilon})}.
$$
Notice that here
$$\Fc_{\omega}[F]\left(\tfrac{x-x_{\eps}}{\varepsilon}\right) =-\Fc_{\omega_{\varepsilon}}[\uzero](x).$$
By applying iteratively Proposition~\ref{prop:une:iteration} with 
$$\begin{cases}
\varphi^{k}(x)= r_{\eps}^k(x),& \forall x\in\partial\domzero,\\
F^{k}(X)= r_{\eps}^k(x),& \forall x=x_{\eps}+\eps X\in\partial\omega_{\eps}, 
\end{cases}$$
we build a sequence of functions $\varphi^k$ defined on $\partial\Omega_{0}$ and $F^k$ defined on $\partial\omega$ such that, for any $N\geq 1$,
\begin{multline}\label{eq.restetildeN}
\ueps(x)=\uzero(x)+\sum_{k=0}^N \varepsilon^k 
\left[{\mathcal G}[\varphi^k](x)+ \Fc_{\omega}[F^k]\left(\tfrac{x-x_{\eps}}{\varepsilon}\right)-\Psi_{0}[F^k]+\frac{{\mathcal G}[\varphi^k](x_{\varepsilon})-\Psi_{0}[F^k]}{w_{\omega}(x_{\eps})-\ln{\varepsilon}} \ \W_{\eps,\omega}(x)\right]\\
 +\varepsilon^{N+1}\ r_{\eps}^{N+1}(x),
\end{multline}
where
\begin{equation*}
\|r_{\eps}^{k+1}\|_{\sL^{\infty}(\partial\domzero \cap \partial\omega_{\varepsilon})} \leq 
C \| r_{\eps}^{k} \|_{\sL^{\infty}(\partial\omega_{\varepsilon})} \leq C^{k+1} \| \uzero\|_{\sL^{\infty}(\partial\omega_{\varepsilon})}.
\end{equation*}
Since $\f\in\sH^{-1+\mu}(\domzero)$ (see \eqref{Laplace-0}), then $\uzero\in \sH^{1+\mu}(\domzero)$ and $\|\uzero\|_{\sL^\infty(\domzero)}\leq c \|\f\|_{\sH^{-1+\mu}(\domzero)}$.\\

An important point to notice is that we obtain estimates of the remainders in the $\sL^\infty$-norm by this method. By properties of harmonic functions (namely, by the maximum principle and Lemma~\ref{lem.GT}), we state that the estimates also hold in the energy norm $\sH^1$ on any compact subset of $\domzero$. Such a restriction also appears in the method developed by M. Dalla Riva and P. Musolino in \cite{DallaRiva2012,DallaRiva2015}. \\

Nevertheless, estimates in the energy norm in the full domain can be obtained following the strategy used in \cite{BDTV09,BD13} where one decomposes the correctors on homogeneous harmonic functions. Using \cite[Proposition 3.2]{BDTV09}, traces of functions are estimated on the singular boundary $\partial\omega_{\varepsilon}$.

\begin{remark}\label{rem.ordre1}
In \cite{BD13}, the support of $\f$ is assumed to be far away from the inclusions. With this assumption, the first term of the expansion can be simplified since the term $\Fc_{\omega}[F^0]\left(\tfrac{x-x_{\eps}}{\varepsilon}\right)-\Psi_{0}[F^0]$ in \eqref{eq.restetildeN} is of the same order of the remainder $\varepsilon\ r_{\eps}^{1}(x)$ and then can be removed. Indeed, using Lemma~\ref{lem.GT}, we have
\begin{align*}
\forall x\in\partial\omega_{\varepsilon},\quad
\left|\Fc_{\omega}[F^0]\left(\tfrac{x-x_{\eps}}{\varepsilon}\right)-\Psi_{0}[F^0]\right|
&= \left|-\uzero(x)+ \frac1{2\pi} \int_{0}^{2\pi} \uzero\left(x_{\varepsilon}+\eps\Tc^{-1} \begin{pmatrix} \cos \theta \\ \sin \theta \end{pmatrix} \right)\, d\theta \right|\\
&\leq \varepsilon \|\nabla\uzero\|_{\sL^\infty(\B(x_{0},\delta/2))}\leq C\varepsilon \|\uzero\|_{\sL^\infty(\B(x_{0},\delta))},\\[5pt]
\forall x\in\partial\domzero,\quad
\left|\Fc_{\omega}[F^0]\left(\tfrac{x-x_{\eps}}{\varepsilon}\right)-\Psi_{0}[F^0]\right|
&\leq \varepsilon\|\uzero\|_{\sL^\infty(\domzero)}.
\end{align*}
In the same way, we also have 
$$ |-\uzero(x_{\varepsilon}) - \Psi_{0}[F^0] | \leq C\varepsilon \|\uzero\|_{\sL^\infty(\domzero)}.$$
Therefore we recover the expansion in \cite{BD13}:
\begin{equation*}
\ueps(x)=\uzero(x)+ \frac{\uzero(x_{\eps})}{w_{\omega}(x_{\eps})-\ln{\varepsilon}} \ \W_{\eps,\omega}(x)
 +\varepsilon\ r_{\eps}^{1}(x).
\end{equation*}
\end{remark}

\section{$N$ inclusions well separated}\label{sec:4}

We come back now to the framework of the introduction, that is we consider $N$ inclusions $(\omega_{\eps}^{i})_{1\leq i\leq N}$ of size $\eps$ centered at points $x_{\eps}^i$. We shall consider two cases depending on the limits of the distance between centers $x_{\varepsilon}^{i}$ when $\varepsilon\rightarrow 0$:
\begin{equation*}
d_{\varepsilon}:=\min_{i\neq j}|x_{\eps}^i-x_{\eps}^j|.
\end{equation*}
\begin{enumerate}
\item The first case is the fixed limit centers: when $\eps\to 0$, the centers tend to $x_{0}^i\in \domzero$ and there exists $C>0$ such that
 \begin{equation*}
d_{\varepsilon} \geq C .
\end{equation*}

\item The second case presents a third scale $\eta(\varepsilon)$ between $\varepsilon$ and $1$ such that
\begin{equation*}
d_{\varepsilon} \geq C \eta(\varepsilon),
\end{equation*}
with a positive constant $C$. This scale $\eta(\varepsilon)$ is assumed to satisfy
\begin{equation}
\label{hypo:eta}
\eta(\varepsilon)\rightarrow 0\qquad \text{ and }\qquad\frac{\eta(\varepsilon)}{\varepsilon}\rightarrow +\infty \qquad\text{ as }\varepsilon\rightarrow 0.
\end{equation}
A typical choice for $\eta(\varepsilon)$ is $\varepsilon^\alpha$ with $\alpha\in[0,1)$ as made in \cite{BDTV09,BD13}. 
\end{enumerate}

Our aim is to apply here the strategy introduced in the case of a single defect. 
For any $1\leq i\leq N$, we associate $w_{i}$ and $\W_{\eps,i}$ as we have introduced $w_{\omega}$ and $\W_{\eps,\omega}$ in Proposition~\ref{prop:une:iteration} in the case of a single inclusion\footnote{For shortness, we replace the index $\omega^i$ by $i$}:
$$\W_{\eps,i}:=\ln|\Tc_{\eps}^i|-w_{i}\qquad \mbox{ with }\qquad w_{i}=\Gc[\ln\beta^i|\cdot-x_{\eps}^i |],$$
with $\beta^i$ the transfinite diameter of the conformal map $\Tc_{\eps}^i$ which sends $\R^2\setminus \overline{\omega_{\eps}^i}$ onto $\R^2\setminus \overline{\B(0,\eps)}$.
We aim at constructing an asymptotic expansion of the solution $\ueps$ of \eqref{Laplace-eps}.

To identify the second term of the asymptotic expansion, we superpose the contribution of each inclusion:
\begin{equation*}
\ueps(x)=\uzero(x)
+ \sum_{i=1}^N \Big( \Fc_{\omega^{i}}[F]\left(\tfrac{x-x_{\eps}^{i}}{\varepsilon}\right)-\Psi_{0,i}\Big)
+ \sum_{i=1}^N
a_{\eps,i}\W_{\eps,i}(x)+\varepsilon\reste1(x),
\end{equation*}
where 
\begin{equation*}
\Fc_{\omega^{i}}[F]\left(\tfrac{x-x_{\eps}^{i}}{\varepsilon}\right)=-\Fc_{\omega_{\varepsilon}^{i}}[\uzero](x)
=\Psi_{0,i}+\mathcal{O}\left(\tfrac{1}{|x|}\right) \qquad\mbox{ as }\quad|x|\to \infty.
\end{equation*}
We look for coefficients $a_{\eps,i}$ such that the remainder $\varepsilon\reste1$ is of smaller order than the first two terms on $\partial\domeps$. \\
$\bullet$ Let us first consider the boundary $\partial\domzero$. By construction, we have 
$$\varepsilon\reste1(x)=
- \sum_{i=1}^N \Big( \Fc_{\omega^{i}}[F]\left(\tfrac{x-x_{\eps}^{i}}{\varepsilon}\right)-\Psi_{0,i}\Big)
-\sum_{i=1}^N a_{\eps,i} \W_{\eps,i}(x)=\O\left(\left(1+\max_{i}|a_{\eps,i}|\right)\eps\right).$$
$\bullet$ Let $1\leq i \leq N$. We have for any $x\in\partial{\omega}_{\eps}^i$
\begin{align*}
\varepsilon\reste1(x) =
& -\uzero(x)-(-\uzero(x) -\Psi_{0,i}) -a_{\eps,i}\W_{\eps,i}(x)\\
&- \sum_{j\neq i} \Big( \Fc_{\omega^{j}}[F]\left(\tfrac{x-x_{\eps}^{j}}{\varepsilon}\right)-\Psi_{0,j}\Big)
	-\sum_{j\neq i} a_{\eps,j}\W_{\eps,j}(x)\\
=& \Psi_{0,i} -a_{\eps,i}\W_{\eps,i}(x)
- \sum_{j\neq i} \Big( \Fc_{\omega^{j}}[F]\left(\tfrac{x-x_{\eps}^{j}}{\varepsilon}\right)-\Psi_{0,j}\Big)
	-\sum_{j\neq i} a_{\eps,j}\W_{\eps,j}(x).
\end{align*}
Let us notice that $\W_{\eps,i}=\ln\eps-w_{i}$ on $\partial{\omega}_{\eps}^i$. 
Furthermore, since $|x-x_{\eps}^j|=\O(d_{\varepsilon})$ on $\partial{\omega}_{\eps}^i$, we have 
$$\Fc_{\omega^{j}}[F]\left(\tfrac{x-x_{\eps}^{j}}{\varepsilon}\right)-\Psi_{0,j} = \O\left(\frac{\varepsilon}{d_{\varepsilon}}\right),\qquad\mbox{ for }j\neq i.$$
Thus
\begin{align*}
\varepsilon\reste1(x) 
=& \Psi_{0,i}-a_{\eps,i}(\ln\eps-w_{i}(x)) +\mathcal{O}\left(\frac{\varepsilon}{d_{\varepsilon}}\right) 
	-\sum_{j\neq i} a_{\eps,j}\left( \ln \left|\eps\Tc^{j}\left(\tfrac{x-x_{\eps}^j}{\eps}\right)\right|-w_{j}(x)\right).
\end{align*}
By a Taylor expansion of the functions $w_{j}$, we deduce
\begin{align*}
\varepsilon\reste1(x) 
=& \Psi_{0,i}-a_{\eps,i}\left(\ln\eps-w_{i}(x_{\eps}^i)+\O(\eps)\right) +\mathcal{O}\left(\frac{\varepsilon}{d_{\varepsilon}}\right)\\
&-\sum_{j\neq i} a_{\eps,j}\left( \ln \left|\eps\Tc^{j}\left(\tfrac{x-x_{\eps}^j}{\eps}\right)\right|-w_{j}(x_{\eps}^i)+\O(\eps)\right)\\
=& \Psi_{0,i}-a_{\eps,i}(\ln\eps-w_{i}(x_{\eps}^i))
	-\sum_{j\neq i} a_{\eps,j}\left( \ln \beta^{j}|x_{\eps}^i-x_{\eps}^j|-w_{j}(x_{\eps}^i)\right)\\
&+\left(1+\max_{j}|a_{\eps,j}|\right)\mathcal{O}\left(\frac{\varepsilon}{d_{\varepsilon}}\right).
\end{align*}
 Then, one cancels the leading terms (up to the order $\varepsilon/{d_{\varepsilon}}$) by solving 
\begin{equation}\label{eq.systMeps}
\Mc_{\eps}
\begin{pmatrix}
a_{\eps,1}\\ \vdots \\ a_{\eps,N}
\end{pmatrix}
= \begin{pmatrix}
\Psi_{0,1}\\ \vdots\\ \Psi_{0,N}
\end{pmatrix},
\end{equation}
with
\begin{equation*}
\Mc_{\eps} = \begin{pmatrix}
\ln\eps-w_{1}(x_{\eps}^1) & & \ln \beta^{j} |x_{\eps}^i-x_{\eps}^j|-w_{j}(x_{\eps}^i) \\ 
&\ddots& \\
\ln \beta^{j} |x_{\eps}^i-x_{\eps}^j| -w_{j}(x_{\eps}^i)
&& \ln\eps-w_{N}(x_{\eps}^N) 
\end{pmatrix}.
\end{equation*}
In the following, we distinguish several configurations according to the behavior of $d_{\eps}$ and the confi\-gurations of the inclusions. 
In each case, we prove that the matrix $\Mc_{\eps}$ is invertible. Then we deduce the existence of coefficients $a_{\eps,i}$ and a good estimate for the remainder $\reste1$. 

\begin{remark}
For a single inclusion ($N=1$), we recover that $a_{\varepsilon}= \frac{\Psi_{0,1}}{\ln\eps-w_{1}(x_{\eps}^1)}$, see \eqref{eq.restetilde1}.
\end{remark}

\subsection{First case: N inclusions at distance $\O(1)$}
Let us first assume that $d_{\eps}=\O(1)$: there exists $c>0$ such that
$$ |x_{\eps}^i-x_{\eps}^j| \in\left[c^{-1} , c\right],\qquad\forall i\neq j.$$
Note that $\ln \beta^{j} |x_{\eps}^i-x_{\eps}^j| -w_{j}(x_{\eps}^i)=\O(1)$, then the matrix $\Mc_{\eps}$ reads
\begin{equation*}
\Mc_{\eps} = \ln\eps\ {\mathcal I}_{N} + \O(1),
\end{equation*}
which is invertible and verifies 
\begin{equation}\label{eq.Mepsinv}
\Mc_{\eps}^{-1}=\frac{1}{\ln\eps}\ {\mathcal I}_{N}+\O\left(\frac{1}{\ln^2\eps}\right).
\end{equation}
Consequently, the coefficients $a_{\eps,j}$ satisfy 
$$\max_{j}|a_{\eps,j}| = \O\left(\frac{1}{\ln\eps}\right).$$
For the asymptotic construction, we do not use the main term of the inverse of $\Mc_{\eps}$ otherwise the remainder would be in $\O(1/\ln\eps)$. Solving exactly \eqref{eq.systMeps} gives convenient coefficients $a_{\eps,j}$ so that the remainder is in $\O(\eps)$:
$$
\ueps=\uzero
+\sum_{i=1}^N \Big( \Fc_{\omega^{i}}[F]\left(\tfrac{\cdot-x_{\eps}^{i}}{\varepsilon}\right)-\Psi_{0,i}\Big)
+ \sum_{i=1}^N a_{\eps,i}\W_{\eps,i}+\O(\eps)\qquad\mbox{ on }\partial\domeps.
$$
In the first sum, there is no interaction but just a superposition of the effect of each inclusion separately. The interaction appears in the coefficients $a_{\eps,i}$ but only in the second order term (in $\O(1/\ln^2\eps)$) as we can see by using \eqref{eq.Mepsinv}.
We can adapt the construction at any order using the inverse $\Mc_{\eps}^{-1}$ to construct a suitable linear combination of the lifting terms and decrease the order of the remainder terms. 

\begin{remark} 
Let us consider $N=2$. Then $\Mc_{\eps}^{-1}$ is given by 
$$\Mc_{\eps}^{-1}=\frac{1}{\delta(\eps)}\begin{pmatrix}
\ln\eps-w_{2}(x_{\eps}^2) & -\ln \beta^{2} |x_{\eps}^1-x_{\eps}^2|+w_{2}(x_{\eps}^1) \\[5pt]
-\ln \beta^{1} |x_{\eps}^1-x_{\eps}^2| +w_{1}(x_{\eps}^2) & \ln\eps-w_{2}(x_{\eps}^2) 
\end{pmatrix},
$$
with 
$$\delta(\eps)=\left(\ln\eps-w_{1}(x_{\eps}^1)\right)\left(\ln\eps-w_{2}(x_{\eps}^2)\right)
-\left(\ln\beta^1|x_{\eps}^1-x_{\eps}^2|-w_{1}(x_{\eps}^2)\right)\left(\ln\beta^2|x_{\eps}^1-x_{\eps}^2|-w_{2}(x_{\eps}^1)\right).
$$
If $\omega^1$ and $\omega^2$ are unit ball, then $\beta^1=\beta^2=1$ and we recover the expressions obtained in \cite[p. 211]{BD13} when $\f$ is supported far away the inclusions. Indeed, Remark~\ref{rem.ordre1} allows to replace $\Psi_{0,i}$ by $-\uzero(x_{\eps}^i)$ and to remove $\Fc_{\omega^{i}}[F]\left(\tfrac{\cdot-x_{\eps}^{i}}{\varepsilon}\right)-\Psi_{0,i}$.
\end{remark}

\subsection{Second case: N inclusions at distance $\O(\eta(\eps))$}
We assume now that the distance between any two inclusions is of order $\eta(\varepsilon)$: there exists $c>0$ such that
$$\eta(\eps)|x_{\eps}^i-x_{\eps}^j| \in\left[c^{-1} , c\right],\qquad\forall i\neq j.$$
Since $\ln\beta^j|x_{\eps}^i-x_{\eps}^j|=\ln \eta(\eps)+\O(1)$ for any $i\neq j$ and $w_{j}(x_{\eps}^i)=\O(1)$ for any $i,j$, the matrix $\Mc_{\eps}$ satisfies
$$\Mc_{\eps} = 
\ln\eps\ {\mathcal I}_{N}+\ln \eta(\eps) (\ {\mathcal H}_{N}-\ {\mathcal I}_{N})+\O(1)= \Mc_{\varepsilon}^0+ \O(1),$$
where $\Mc_{\varepsilon}^0=(\ln\varepsilon -\ln\eta(\varepsilon)) \ {\mathcal I}_{N}+\ln\eta(\varepsilon) \ {\mathcal H}_{N}$ and $\ {\mathcal H}_{N}$ is the square matrix of size $N$ with every coefficients equal to $1$. Since the rank of ${\mathcal H}_{N}$ is one, there exists an orthogonal matrix $P$ such that 
\begin{eqnarray*}\ {\mathcal H}_{N}=
\begin{pmatrix}
1 & \ldots & 1\\
\vdots&&\vdots \\
1 &\ldots& 1
\end{pmatrix}\
= P \begin{pmatrix}
N & & & 0\\
 & 0& & \\
& &\ddots & \\
0 && & 0
\end{pmatrix} P^{-1}.
\end{eqnarray*}
Thus
$$ \Mc_{\varepsilon}^0 = P
\begin{pmatrix}
\ln\varepsilon +(N-1)\ln\eta(\varepsilon) & 0\\
0 &(\ln\varepsilon -\ln\eta(\varepsilon)) \ {\mathcal I}_{N-1} \\
\end{pmatrix} P^{-1},$$
which is clearly invertible for $\varepsilon$ small. The inverse is
\begin{align*}
(\Mc_{\varepsilon}^0)^{-1}
&= P\begin{pmatrix}
\frac{1}{\ln\varepsilon +(N-1)\ln\eta(\varepsilon) } & 0\\
0 &\frac{1}{\ln\varepsilon -\ln\eta(\varepsilon) }\ {\mathcal I}_{N-1}
\end{pmatrix} P^{-1}\\
&= \cfrac{1}{\ln\varepsilon -\ln\eta(\varepsilon) } {\mathcal I}_{N} +\frac{1}N\left[\cfrac{1}{\ln\varepsilon +(N-1)\ln\eta(\varepsilon) } -\cfrac{1}{\ln\varepsilon -\ln\eta(\varepsilon) }\right]{\mathcal H}_{N}.
\end{align*}
Consequently, $\Mc_{\eps}$ is invertible and, by Neumann series, we can remark
\begin{equation*}
\Mc_{\eps}^{-1}
=\frac{1}{\ln\eps}\left(\frac{1}{1 -\frac{\ln\eta(\varepsilon)}{\ln\varepsilon} } {\mathcal I}_{N} +\frac{1}N\left[\frac{1}{1 +(N-1)\frac{\ln\eta(\varepsilon)}{\ln\varepsilon} } -\frac{1}{1 -\frac{\ln\eta(\varepsilon)}{\ln\varepsilon} }\right]{\mathcal H}_{N}\right) + \O\left(\frac{1}{\ln^2\eps}\right).
\end{equation*}
The asymptotic expansion of $\ueps$ is thus given by
\begin{multline*}
\ueps(x) = \uzero(x)+\sum_{i=1}^N \Big( \Fc_{\omega^{i}}[F]\left(\tfrac{x-x_{\eps}^{i}}{\varepsilon}\right)-\Psi_{0,i}\Big)\\
+\left\langle\Mc_{\eps}^{-1}\begin{pmatrix}
\Psi_{0,1}\\ \vdots\\ \Psi_{0,N}
\end{pmatrix},\begin{pmatrix}
\W_{\eps,1}(x)\\ \vdots\\ \W_{\eps,N}(x)
\end{pmatrix}
\right\rangle+\O\left(\frac{\eps}{\eta(\varepsilon)\ln{\varepsilon}}\right),
\end{multline*}
where we have expressed the coefficients $a_{\eps,i}$ according to the resolution of the system~\eqref{eq.systMeps}.

In order to compute the leading corrector, one has to consider the limit of $\ln{\varepsilon}\Mc_{\varepsilon}^{-1}$ when $\varepsilon\rightarrow 0$.
The limit matrix depends on the ratio $\ln{\eta(\varepsilon)}/\ln{\varepsilon}$. If this ratio has a finite limit $l$, then $l\in[0,1)$ according to \eqref{hypo:eta} and 
\begin{equation}\label{eq.Meps-1alphabis}
\Mc_{\eps}^{-1}\sim \frac{1}{\ln\eps}\left(\frac{1}{1-l} {\mathcal I}_{N} +\frac{1}N\left[\frac{1}{1 +(N-1)l} -\frac{1}{1-l}\right]{\mathcal H}_{N}\right).
\end{equation}
In expression \eqref{eq.Meps-1alphabis}, we see that the situation is now more complex than in the previous case \eqref{eq.Mepsinv} since all the defects interact at leading order $\O(1/\ln\eps)$. \\
As an example, if $\eta(\varepsilon)=\varepsilon^\alpha$ and $N=2$, then $l=\alpha$ and we have
\begin{multline*}
\ueps(x) 
= \uzero(x) +\sum_{i=1}^2 \Big( \Fc_{\omega^{i}}[F]\left(\tfrac{x-x_{\eps}^{i}}{\varepsilon}\right)-\Psi_{0,i}\Big)\\
+\frac{1}{(1-\alpha^2)\ln\eps}\left\langle\begin{pmatrix}1 & -\alpha\\ -\alpha & 1
\end{pmatrix}
\begin{pmatrix}
\Psi_{0,1}\\ \Psi_{0,2}
\end{pmatrix},\begin{pmatrix}
\W_{\eps,1}(x)\\ \W_{\eps,2}(x)
\end{pmatrix}
\right\rangle+\O\left(\frac{1}{\ln^2(\varepsilon)}\right).
\end{multline*}
If $\f$ is supported far away the inclusions, we obtain, using Remark~\ref{rem.ordre1}, the expansion of \cite[Section 4.2]{BD13}:
\[\ueps(x)
= \uzero(x) -\frac{\uzero(x^1_{0})}{(1+\alpha)\ln\eps} (\W_{\eps,1}(x)+ \W_{\eps,2}(x))
+\O\left(\frac{1}{\ln^2(\varepsilon)}\right).
\]

\subsection{More complex geometrical settings}

There are various situations which can be dealt by our approach. In order to show that the method is versatile, let us shortly consider two more complex geometrical setting. 

\subsubsection{Two inclusions at distance $\O(\eps^\alpha)$ and the others at distance $\O(1)$}
We assume $N\geq 3$ and, with a possible renumbering, that the first two inclusions are at distance $\O(\eps^\alpha)$ whereas the others are at distance $\O(1)$, namely there exists $c>0$
$$\eps^{-\alpha}|x_{\eps}^1-x_{\eps}^2|\in [c^{-1}, c]\qquad\mbox{ and }\qquad
|x_{\eps}^i-x_{\eps}^j|\in [c^{-1}, c],\quad\forall i<j, (i,j)\neq(1,2).
$$
In this case, $\Mc_{\eps}$ has the expansion
\begin{equation*}
\Mc_{\eps}=\ln\eps\begin{pmatrix}
1 & \alpha & &0\\
\alpha & 1 & &&\\
 & & \ddots & \\
 0 && & 1 
\end{pmatrix} +\O(1),
\end{equation*}
with inverse satisfying
\begin{equation*}
\Mc_{\eps}^{-1}=\frac 1{\ln\eps(1-\alpha^2)}\begin{pmatrix}
1 & -\alpha & &0\\
-\alpha & 1 & &&\\
 & & \ddots & \\
 0 && & 1 
\end{pmatrix} +\O\left(\frac 1{\ln^2\eps}\right).
\end{equation*}
We notice that only the first two defects interact at the leading order $\O(1/\ln\eps)$ while the interaction involving the other defects is postponed at order $\O(1/\ln^2\eps)$. 
\begin{remark}
If we use the expansion of $\Mc_{\eps}^{-1}$, we have the {\it weak} asymptotic expansion for $\ueps$ with a worse remainder:
\begin{equation*}
\begin{split}
\ueps=& \uzero 
+\sum_{i=1}^N \Big( \Fc_{\omega^{i}}[F]\left(\tfrac{.-x_{\eps}^{i}}{\varepsilon}\right)-\Psi_{0,i}\Big)\\
&+\frac{\Psi_{0,1}-\alpha \Psi_{0,2}}{\ln\eps(1-\alpha^2)} \W_{\eps,1}
+\frac{\Psi_{0,2}-\alpha \Psi_{0,1}}{\ln\eps(1-\alpha^2)} \W_{\eps,2}
+\sum_{j=3}^N\frac{\Psi_{0,j}}{\ln\eps(1-\alpha^2)} \W_{\eps,j}
+\O\left(\frac{1}{\ln^2\eps}\right).
\end{split}
\end{equation*}

This expansion sheds light the fact that at the first order, there is only an interaction between the first two inclusions.
\end{remark}

\subsubsection{Particular case with 3 inclusions}
Let us analyze a last situation with three inclusions and three scales. We consider 
$0< \beta \leq \alpha<1$ and $c>0$ such that
$$
\eps^{-\alpha}|x_{\eps}^1-x_{\eps}^2| \in [c^{-1},c],\qquad
\eps^{-\beta}|x_{\eps}^i-x_{\eps}^3| \in [c^{-1},c],\quad\mbox{ for }i=1,2.
$$
The matrix $\Mc_{\eps}$ is such that
\begin{equation*}
\Mc_{\eps}=\ln\eps\ M_{\alpha,\beta} +\O(1)\qquad\mbox{ with }\qquad M_{\alpha,\beta}=\begin{pmatrix}
1 & \alpha & \beta\\
\alpha & 1 & \beta\\
\beta & \beta & 1 
\end{pmatrix}.
\end{equation*}
Computing the determinant of $M_{\alpha,\beta}$, we have
$$\det M_{\alpha,\beta}=(\alpha-1)(2\beta^2-\alpha-1),$$
which does not vanish when $0<\beta\leq\alpha<1$ (since the trinomial function $\alpha\mapsto 2\alpha^2-\alpha-1$ is negative on $(0,1)$). 
Consequently $\Mc_{\eps}$ is invertible and the inverse of the leading term is explicit
$$M_{\alpha,\beta}^{-1} = P 
 \begin{pmatrix}
\frac{2}{\alpha+2+\sqrt{\alpha^2+8\beta^2}} &0 & 0\\
0 & \frac{2}{\alpha+2-\sqrt{\alpha^2+8\beta^2}} & 0\\
0 & 0 & \frac 1{1-\alpha}
\end{pmatrix}P^{-1},$$
where $P$ is the orthogonal matrix given by
$$
P=\frac{1}{\sqrt 6}\begin{pmatrix} \sqrt 2 & \sqrt 3 & 1\\ \sqrt 2 & -\sqrt 3 & 1\\ \sqrt 2 & 0 & -2
\end{pmatrix}.
$$
Consequently, there is a complete interaction between the three defects.\\
When $\alpha=\beta$, we recover the main term in \eqref{eq.Meps-1alphabis} for $N=3$.
\\

\noindent {\bf Aknowledgements.}
The authors are partially supported by the ANR (Agence Nationale de la Recher\-che), projects {\sc Aramis} n$^{\rm o}$ ANR-12-BS01-0021 (for the first two authors) and {\sc DYFICOLTI} n$^{\rm o}$ ANR-13-BS01-0003-01 (for the third one).

\def\cftil#1{\ifmmode\setbox7\hbox{$\accent"5E#1$}\else
  \setbox7\hbox{\accent"5E#1}\penalty 10000\relax\fi\raise 1\ht7
  \hbox{\lower1.15ex\hbox to 1\wd7{\hss\accent"7E\hss}}\penalty 10000
  \hskip-1\wd7\penalty 10000\box7} \newcommand{\noopsort}[1]{}
  \let\v\vv\def\cprime{$'$}

\adrese

\end{document}